\documentclass[a4,11pt]{amsart}

\usepackage{latexsym}
\usepackage{amsmath}
\usepackage{amssymb}
\usepackage{amsthm}
\usepackage[mathcal]{eucal}
\usepackage{amsmath, amssymb,bbm,enumerate}
\usepackage{mathrsfs}
\usepackage{enumerate}
\usepackage{color}
\allowdisplaybreaks

 \makeatletter \@addtoreset{equation}{section}

\makeatother \makeatletter

\newtheorem{lem}{Lemma}[section]
\newtheorem{cor}[lem]{Corollary}
\newtheorem{teo}[lem]{Theorem}
\newtheorem{os}[lem]{Remark}

\newtheorem{prop}[lem]{Proposition}

\newcommand{\R}{\mathbb{R}}

\newcommand{\C}{\mathbb{C}}

\newcommand{\N}{\mathbb{N}}

\newcommand{\one}{\mathbbm{1}}

\newcommand{\ov}{\overline}

\newcommand{\supp}{\emph{supp\,}}

\allowdisplaybreaks

\begin{document}
\title{Schr\"odinger type operators with unbounded diffusion and potential terms}
\author{A. Canale}
\address{Dipartimento di Matematica, Universit\`a degli Studi di Salerno, Via Giovanni Paolo II, 132, I 84084 FISCIANO (Sa), Italy.}
\email{acanale@unisa.it}
\author{A. Rhandi}
\address{Dipartimento di Ingegneria dell'Informazione, Ingegneria Elettrica e Matematica Applicata, Universit\`a degli Studi di Salerno, Via Giovanni Paolo II, 132, I 84084 FISCIANO (Sa), Italy.}
\email{arhandi@unisa.it}
\author{C. Tacelli}
\address{Dipartimento di Ingegneria dell'Informazione, Ingegneria Elettrica e Matematica Applicata, Universit\`a degli Studi di Salerno, Via Giovanni Paolo II, 132, I 84084 FISCIANO (Sa), Italy.}
\email{ctacelli@unisa.it}

\maketitle

\begin{abstract}
We prove that the realization $A_p$ in $L^p(\R^N),\,1<p<\infty$, of the Schr\"odinger type operator $A=(1+|x|^{\alpha})\Delta-|x|^{\beta}$ with domain $D(A_p)=\{u\in W^{2,p}(\R^N): Au\in L^p(\R^N)\}$ generates a strongly continuous analytic semigroup provided that $N>2,\,\alpha >2$ and $\beta >\alpha -2$. Moreover this semigroup is consistent, 
irreducible, immediately compact and ultracontractive.
\end{abstract}

\section{Introduction}
In this paper we study the generation of analytic semigroups in $L^p$-spaces of Schr\"odinger type operators of the form
\begin{equation}\label{eq:operatore-A}
 Au(x)=a(x)\Delta u(x)-V(x)u(x), \quad x\in\R^N ,
\end{equation}
where $a(x)=1+|x|^\alpha$ and $V(x)=|x|^\beta$ with $\alpha >2$ and $\beta>\alpha -2$. We investigate also spectral properties of such semigroups.
In the case when $\alpha\in [0,2]$ and $\beta \ge 0$,
generation results of analytic semigroups for suitable realizations $A_p$ of the operator $A$ in $L^p(\R^N)$ have been proved in \cite{lor-rhan}.

For $\beta=0$ and $\alpha>2$, the generation results depend upon $N$ as it is
proved in \cite{met-spi2}. More specifically, if $N=1,2$ no realization of $A$ in $L^p(\R^N)$ generates a strongly continuous (resp. analytic) semigroup.
The same happens if $N\ge 3$ and $p\le N/(N-2)$.
On the other hand, if $N\ge 3$ and $p>N/(N -2)$, then the maximal realization $A_p$ of the operator $A$ in $L^p(\R^N)$
generates a positive analytic semigroup, which is also contractive if $\alpha\geq(p - 1)(N - 2)$.

Generation results concerning the case where $\beta=0$ and with drift terms of the form $|x|^{\alpha -2}x$ were obtained recently in \cite{met-spi-tac}.

Here we consider the case where $\alpha >2$ and assume that $N>2$. Let us
denote by $A_p$ the realization of $A$ in $L^p(\R^N)$ endowed with its maximal domain
\begin{equation}
D_{p,max}(A)=\{ u\in L^p(\R^N) \cap W^{2,p}_{loc}(\R^N)\,:\;Au\in L^p(\R^N) \}.
\end{equation}
In the main result of the paper we prove that, for any $1<p<\infty$, the realization $A_p$ of $A$ in $L^p(\R^N)$, with domain
\begin{eqnarray*}
D(A_p)=\left\{u\in W^{2,p}(\R^N): Au\in L^p(\R^N)\right\},
\end{eqnarray*}
generates a positive strongly continuous and analytic semigroup $(T_p(t))_{t\ge 0}$ for any $\beta>\alpha -2$.
This semigroup is also consistent, irreducible, immediately compact and ultracontractive.

The paper is structured as follows. In Section 2 we study the invariance of $C_0(\R^N)$ under the semigroup generated by $A$ in $C_b(\R^N)$ and show its compactness. In Section 3 we use reverse H\"older classes and some results in \cite{shen-1995} to study the solvability of the elliptic problem in $L^p(\R^N)$. Then, in Section 4 we prove the generation results.

\bigskip\noindent
{\large \bf Notation.}
For any $k\in\N\cup\{\infty\}$ we denote by $C^{k}_c(\R^N)$ the set of all
functions $f:\R^N\to\R$ that are continuously differentiable in $\R^N$ up to $k$-th order
and have compact support (say ${\rm supp}(f)$). The space $C_b(\R^N)$ is the set of all
bounded and continuous functions $f:\R^N\to
\R$, and we denote by $\|f\|_{\infty}$ its sup-norm, i.e., $\|f\|_{\infty}=\sup_{x\in\R^N}|f(x)|$. We use also the space $C_0(\R^N):=\{f\in C_b(\R^N): \lim_{|x|\to \infty}f(x)=0\}.$
If $f$ is smooth enough we set
\begin{eqnarray*}
|\nabla f(x)|^2=\sum_{i=1}^N|D_if(x)|^2,\qquad
|D^2f(x)|^2=\sum_{i,j=1}^N|D_{ij}f(x)|^2.
\end{eqnarray*}
For any $x_0\in\R^N$ and any $r>0$ we denote by $B(x_0,r)\subset\R^N$ the open ball, centered at $x_0$ with radius $r$.
We simply write $B(r)$ when $x_0=0$.
The function $\chi_E$ denotes the characteristic function of the (measurable) set $E$,
i.e., $\chi_E(x)=1$ if $x\in E$, $\chi_E(x)=0$ otherwise.

For any $p\in [1,\infty)$ and any positive measure $d\mu$, we simply write $L^p_{\mu}$ instead of $L^p(\R^N,d\mu)$.
The Euclidean inner product in $L^2_{\mu}$ is denoted by $(\cdot,\cdot)_{\mu}$.
In the particular case when $\mu$ is the Lebesgue measure, we keep the classical notation $L^p(\R^N)$ for any $p\in [1,\infty)$.
Finally, by $x\cdot y$ we denote the Euclidean scalar product of the vectors $x,y\in\R^N$.

\section{Generation of semigroups in $C_0(\R^N)$}\label{sc:c0}

In this section we recall some properties of the elliptic and parabolic problems associated with  $A$
in $C_b(\R^N)$. We prove the existence of a Lyapunov function for $A$ in the case where $\alpha >2$ and $\beta >\alpha -2$. This implies the uniqueness of the solution semigroup $(T(t))_{t\ge 0}$ to the associated parabolic problem. Using a dominating argument, we show that $T(t)$ is compact and $T(t)C_0(\R^N)\subset C_0(\R^N)$.

First, we endow $A$ with its maximal domain in $C_b(\R^N)$
\[
D_{max}(A)=\{u\in C_b(\R^N)\cap W^{2,p}_{loc}(\R^N),\quad 1\le p<\infty:\ Au\in C_b(\R^N)\}.
\]
Then, we consider for any $\lambda >0$ and $f\in C_b(\R^N)$ the elliptic equation
\begin{equation}\label{eq:C0-elliptic-problem}
\lambda u -Au=f.
\end{equation}
It is well-known that equations \eqref{eq:C0-elliptic-problem} admit at least a solution in $D_{max}(A)$
(see \cite[Theorem 2.1.1]{ber-lor}). A solution is obtained as follows.\\
Take the unique solution to the Dirichlet problem associated with $\lambda -A$
into the balls $B(0,n)$ for $n\in \N$. Using an Schauder interior estimates one can prove that
the sequence of solutions so obtained converges to a solution $u$ of
\eqref{eq:C0-elliptic-problem}. It is also known that a solution to \eqref{eq:C0-elliptic-problem} is in general not unique. The solution $u$,
which we obtained by approximation, is nonnegative whenever $f\ge 0$.

As regards the parabolic problem
\begin{equation}     \label{problem}
\left\{
\begin{array}{ll}
u_t(t,x)=Au(t,x)& x\in\R^N,\ t>0, \\
u(0,x)=f(x)   &x\in\R^N\;, \end{array}\right.
\end{equation}
where $f\in C_b(\R^N)$,
it is well-known that one can associate a semigroup $(T(t))_{t\ge 0}$ of bounded operator in $C_b(\R^N)$
such that $u(t,x)=T(t)f(x)$ is a solution of \eqref{problem} in the following sense:
$$u\in C([0,+\infty)\times \R^N)\cap C_{loc}^{1+\frac{\sigma}{2},2+\sigma}((0,+\infty)\times \R^N)$$
and $u$ solves \eqref{problem} for any $f\in C_b(\R^N)$ and some $\sigma \in (0,1)$.
Uniqueness of solutions to \eqref{problem} in general is not guaranteed. Moreover the
semigroup $(T(t))_{t\ge 0}$ is not strongly continuous in $C_b(\R^N)$ and do not preserve in general the space $C_0(\R^N)$. 
For more details we refer to \cite[Section 2]{ber-lor}.
%
%

Uniqueness is obtained if there
exists a positive function  $\varphi(x)\in C^2(\R^N)$, called {\it Lyapunov function}, such that $\lim_{|x|\to \infty}\varphi(x)=+\infty$
and $A\varphi-\lambda \varphi\leq 0$ for some $\lambda>0$.

\begin{prop}
Let $N>2$, $\alpha>2$ and $\beta>\alpha-2$. Let $\varphi=1+|x|^\gamma$ where $\gamma>2$ then there existence constant $C>0$ such that
\[
 A\varphi \leq C \varphi
\]
\end{prop}
\begin{proof}
An easy computation gives
\[
 A\varphi=\gamma (N+\gamma-2)(1+|x|^\alpha)|x|^{\gamma-2}-(1+|x|^\gamma)|x|^\beta\;,
\]
then since $\beta>\alpha-2$ there existence $C>0$ such that
\[
\gamma (N+\gamma-2)(1+|x|^\alpha)|x|^{\gamma-2}\leq (1+|x|^\gamma)|x|^\beta+C(1+|x|^\gamma).
\]
\end{proof}

Then we can assert that problem $\eqref{problem}$ admits an unique solution in $C([0,\infty)\times \R^N)\cap C^{1,2}((0,\infty)\times \R^N)$
and problem \eqref{eq:C0-elliptic-problem} admits an unique solutions in $D_{max}(A)$.

In order to investigate the compactness of the semigroup and the invariance of $C_0(\R^N)$ we
check the behaviour of  $T(t)\one$. We use the following result (see \cite[Theorem 5.1.11]{ber-lor}).
\begin{teo}
Let us fix $t>0$. Then $T(t)\one\in C_0(\R^N)$ if and only if $T(t)$ is compact and $C_0(\R^N)$ is invariant for $T(t)$.
\end{teo}

Let $A_0$ be the operator defined by $A_0:=a(x)\Delta$.
By \cite[Example 7.3]{met-wack} or \cite[Proposition 2.2 (iii)]{met-spi2}, we have that
the minimal semigroup $(S(t))$ is generated by $(A_0,D_{max}(A_0))\cap C_0 (\R^N)$. Moreover
the resolvent and the semigroup map $C_b(\R^N)$ into $C_0 (\R^N )$ and are compact.

Set $v(t,x)=S(t)f(x)$ and $u(t,x)=T(t)f(x)$ for $t>0,\,x\in \R^N$ and $0\le f\in C^b(\R^N)$. Then the function $w(t,x)=v(t,x)-u(t,x)$ solves
$$\left\{\begin{array}{ll}
w_t(t,x)=A_0w(t,x)+V(x)u(t,x),\quad t>0,\\
w(0,x)=0,\quad x\in \R^N.
\end{array} \right.$$
So, applying \cite[Theorem 4.1.3]{ber-lor}, we have $w\ge 0$ and hence $T(t)\le S(t)$. Thus, $T(t)\one \in C_0(\R^N)$, since $S(t)\one \in C_0(\R^N)$ for any $t>0$ (see \cite[Proposition 2.2 (iii)]{met-spi2}). Thus, $T(t)$ is compact and $C_0(\R^N)$ is invariant for $T(t)$ (cf. \cite[Theorem 5.1.11]{ber-lor}).
Then we have proved the following proposition
\begin{prop}\label{pr:semigroup-C0}
The semigroup $(T(t))$ is generated by $(A,D_{max}(A))\cap C_0 (\R^N)$, maps $C_b(\R^N)$ into $C_0(\R^N)$, and is compact.
\end{prop}

\section{Solvability of the elliptic problem in $L^p(\R^N)$}

In this section we study the existence and uniqueness of the elliptic problem $\lambda u-A_pu=f$ for a given $f\in L^p(\R^N),\,1<p<\infty$ and $\lambda \ge 0$. Let us consider first the case $\lambda =0$.

We note that the equation $(1+|x|^\alpha)\Delta u-Vu=f$ is equivalent to the equation
\[
 \Delta u-\frac{V}{1+|x|^\alpha}u=\frac {f}{1+|x|^{\alpha}}=:\tilde f.
\]
Therefore we focus our attention to the $L^p$-realization $\tilde A_p$ of the Schr\"odinger operator
\[
\tilde A =\Delta -\frac{V}{1+|x|^\alpha}=\Delta - \tilde V\;.
\]
Here $0\le \tilde V \in L^1_{\rm loc}(\R^N)$. So, by standard results, it follows that $0\in \rho(\tilde A_p)$ and
\begin{equation}\label{eq:A-tilde-inversa}
(-\tilde A_p)^{-1}\tilde f(x)=\int _{\R^N}G(x,y)\tilde f(y)dy,
\end{equation}
where $G$ denotes the Green function of $\tilde A_p$ which is given by its heat kernel $\tilde p$
$$G(x,y)=\int_0^\infty \tilde p(t,x,y)\,dt.$$
Thus the function
$u=\int _{\R^N}G(x,y)\frac {f(y)}{1+|y|^\alpha}dy \in D_{p,max}(A)$ and
solves $A_pu=f$ for every $f\in L^p(\R^N)$.
So we have to study the operator
\begin{equation}\label{eq:operatore-L}
u(x)=Lf(x):=\int_{\R^N}G(x,y)\frac{f(y)}{1+|y|^\alpha}dy.
\end{equation}
 To this purpose, we use the bounds of $G(x,y)$ obtained in \cite{shen-1995}
when the potential of $\tilde A_p$ belongs to the reverse H\"older class $B_q$ for some $q\geq N/2$ .

We recall that a nonnegative locally $L^q$-integrable function $V$ on $\R^N$ is
said to be in $B_q,\,1 < q < \infty$, if there exists $C > 0$ such that the reverse
H\"older inequality
\[
\left( \frac{1}{|B|}\int_BV^q(x) dx \right) ^{1/q}\leq C\left( \frac{1}{|B|}\int_BV(x) dx \right)
\]
holds for every ball $B$ in $\R^N$. A nonnegative function $V\in L_{\rm loc}^\infty(\R^N)$ is in
$B_\infty $ if
\[
\|V\|_{L^\infty(B)}\leq C\left( \frac{1}{|B|}\int_BV(x) dx \right)
\]
for any ball $B$ in $\R^N$.

One can easily verify that
\begin{equation}\label{shen-B}
\tilde V\in \left\{ \begin{array}{ll}
B_\infty \quad \hbox{\ if }\beta -\alpha \ge 0 \\
B_q \quad \,\,\hbox{\ if }\beta -\alpha>-\frac{N}{q} \\
B_{\frac{N}{2}} \quad \hbox{\ if }\beta -\alpha >-2 \\
B_N \quad \hbox{\ if }\beta -\alpha >-1
\end{array} \right.
\end{equation}
for some $q>1$. So, it follows from \cite[Theorem 2.7]{shen-1995} that, if $\beta -\alpha >-2$ then for any $k>0$
there is some constant $C_k>0$ such that for any $x,\,y\in \R^N$
\begin{equation}\label{eq:stima-G-shen}
|G(x,y)|\leq\frac{C_k}{\left( 1+m(x)|x-y| \right)^k}\cdot \frac{1}{|x-y|^{N-2}},
\end{equation}
where the function $m$ is defined by
\begin{equation}\label{eq:def-m}
\frac{1}{m(x)}:=\sup_{r>0}\left\{r\,:\,\frac{1}{r^{N-2}} \int_{B(x,r)}\tilde V(y)dy\leq 1 \right\},\quad x\in R^N.
\end{equation}
Due to the importance of the auxiliary function $m$ we give a lower bound.

\begin{lem}\label{pr:mx-beta-le-alpha}
Let $\alpha-2<\beta<\alpha$. There exists $C=C(\alpha,\beta,N)$ such that
\begin{equation}\label{eq:stima-m-beta-le-alpha}
m(x)\geq  C \left( 1+|x|\right)^{\frac{\beta-\alpha}{2}}.
\end{equation}
\end{lem}
\begin{proof} Fix $x\in \R^N$, and set $f_x(r)=\frac{1}{r^{N-2}}\int _{B(x,r)}\tilde V(y)dy,\,r>0$.
Since $\tilde V\in B_{N/2}$ implies $V\in B_q$ for some $q>\frac{N}{2}$, by \cite[Lemma 1.2]{shen-1995}, we have
\[
\lim_{r \to 0}f_x(r)=0\,\text{ and } \lim_{r \to \infty}f_x(r)=\infty .
\]
Thus, $0<m(x)<\infty$.
\\
In order to estimate $\frac{1}{m(x)}$ we need to find $r_0=r_0(x)$ such that $r \in [r_0,\infty[$ implies $f_x(r)\geq 1$. In this case we will have  $\frac{1}{m(x)}\leq r_0$.

Since $\tilde V\in B_{N/2}$, there exists a constant $C_1$ depending only $\alpha,\beta,N$ such that
\[
\left( \frac{1}{|B|}\int_B\tilde V^{N/2}(y) dy \right) ^{2/N}\leq C_1\left( \frac{1}{|B|}\int_B\tilde V(y)\, dy \right)
\]
for any ball $B$ in $\R^N$. Then we have
\begin{eqnarray*}
f_x(r) &=& \sigma_{N-1} r^2\frac{1}{|B(x,r)|} \int_{B(x,r)}\tilde V(y) dy\\
&\geq & \frac{\sigma_{N-1} r^2}{C_1}
\left( \frac{1}{|B(x,r)|}\int_{B(x,r)}\tilde V(y)^{N/2} dy \right) ^{2/N}\\
&=& \frac{\sigma_{N-1}^{1-2/N}}{C_1}\left( \int_{B(x,r)}\tilde V(y)^{N/2} dy \right) ^{2/N}.
\end{eqnarray*}
Hence, if
\begin{equation}\label{eq:intbv-1}
\int_{B(x,r)}\tilde V(y)^{N/2} dy  - C_2\geq 0\,,
\end{equation}
then $f_x(r)\ge 1$, where $C_2=C_2(\alpha,\beta,N)=\frac{C^{N/2}_1}{\sigma_{N-1}^{N/2-1}}$.
Note that $\tilde V\geq \tilde V^*$ in $\R^N\setminus B(0,1)$ with $\tilde V^*(x)=\frac{1}{2} |x|^{\beta -\alpha}$. Hence,
\begin{eqnarray}\label{eq:int_BV-2}
\int_{B(x,r)}\tilde V(y)^{N/2} dy &\geq & \int_{B(x,r)\setminus B(0,1)}\tilde V(y)^{N/2} dy\geq \int_{B(x,r)\setminus B(0,1)}\tilde V^*(y)^{N/2} dy \nonumber\\
&=& \int_{B(x,r)}\tilde V^*(y)^{N/2}dy-\int_{B(x,r)\cap B(0,1)}\tilde V^*(y)^{N/2}dy \nonumber \\
&\geq & \int_{B(x,r)}\tilde V^*(y)^{N/2}dy-\int_{B(0,1)}\tilde V^*(y)^{N/2}dy\nonumber\\
&=& \int_{B(x,r)}(\tilde V^*)(y)^{N/2}dy-\frac{2^{1-N/2}\sigma_{N-1}}{N(2-\alpha+\beta)}\nonumber \\
&\geq & \sigma_{N-1}r^N\inf_{B(x,r)}(\tilde V^*)^{N/2}-C_3(\alpha,\beta,N)\\
&=& \sigma_{N-1}\frac{r^N}{\left( |x|+r \right)^{\frac{\alpha-\beta}{2}N}}-C_3.
\end{eqnarray}
Let $\eta=\frac{\alpha-\beta}{2}<1$ and $\delta>0$ a parameter to be choose later, and set
\[
 r_0=\delta (1+|x|)^\eta\;.
\]
By \eqref{eq:int_BV-2} condition \eqref{eq:intbv-1}  became
\begin{eqnarray*}
\int_{B(x,r_0)}\tilde V(y)^{N/2} dy  - C_2 &\geq & \sigma_{N-1}\frac{r_0^N}{\left( |x|+r_0 \right)^{\frac{\alpha-\beta}{2}N}}-C_2-C_3\\
&=& \sigma_{N-1} \frac{\delta^N (1+|x|)^{\eta N}}{\left(|x|+\delta (1+|x|)^\eta\right)^{\frac{\alpha-\beta}{2}N}}-C_4\\
&\geq & \sigma_{N-1} \frac{\delta^N (1+|x|)^{\eta N}}{\left(1+|x|+\delta (1+|x|)^\eta\right)^{\frac{\alpha-\beta}{2}N}}-C_4\\
&\geq & \sigma_{N-1} \frac{\delta^N (1+|x|)^{\eta N}}{\left((\delta+1) (1+|x|)\right)^{\frac{\alpha-\beta}{2}N}}-C_4\\
&=&
  \sigma_{N-1} \left( \frac{\delta }{(1+\delta)^{\frac{\alpha-\beta}{2}}}\right)^{N}-C_4\,.
\end{eqnarray*}
Since $\frac{\alpha-\beta}{2}<1$ we can choose $\delta >0$ such that $\sigma_{N-1} \left( \frac{\delta }{(1+\delta)^{\frac{\alpha-\beta}{2}}}  \right)^{N}-C_4\geq 0\,.$

So, \eqref{eq:intbv-1} is satisfied for $r=r_0$ and hence it is satisfied for any $r>r_0$. Thus, $f_x(r)>1$ for $r>r_0$, and, hence, $\frac{1}{m(x)}\leq r_0=\delta (1+|x|)^\eta$.
\end{proof}

The same lower bound holds in the case $\beta\geq \alpha$ as the following lemma shows.

\begin{lem}\label{pr:stim-m}
Let $\beta\geq \alpha$. There exists $C=C(\alpha,\beta,N)$ such that
\begin{equation}\label{eq:stima-m-every-beta}
m(x) \geq  C \left( 1+|x| \right)^{\frac{\beta-\alpha}{2}}.
\end{equation}
\end{lem}
\begin{proof}
From  \cite[Lemma 1.4 (c)]{shen-1995},
there exist $C_1>0$ and $0<\eta_0<1$ such that
\[
 m(x)\geq \frac{C_1 m(y)}{\left(1+|x-y|m(y)\right)^{\eta_0}}\;.
\]
In particular,
\[
m(x)\geq \frac{C_1 m(0)}{\left(1+|x|m(0)\right)^{\eta_0}},
\]
where
$
\frac{1}{m(0)}=\sup_{r>0}\left\{r\,:\,f_0(r)\leq 1 \right\}
$ with
\[f_0(r)=\frac{1}{r^{N-2}} \int_{B(0,r)}\frac{|z|^\beta}{1+|z|^\alpha}dz
=\frac{\sigma_{N-1}}{r^{N-2}}\int_0^r\frac{\rho^{\beta+N-1}}{1+\rho^\alpha}d\rho \;.
\]
We have $\frac{\sigma_{N-1}}{(\beta+N)(1+r^\alpha)}r^{\beta+2}\leq f_0(r)\leq \frac{\sigma_{N-1}}{\beta+N}r^{\beta+2}$.
Since $\beta>0$
and $\beta-\alpha+2>0$ it follows that $\lim_{r\to 0}f_0(r)=0$ and $\lim_{r\to \infty}f_0(r)=\infty$.
Consequently,  $$0<\sup_{r>0}\left\{r\,:\,f_0(r)\leq 1 \right\}<\infty $$ and, hence,
$m(0)=C_2$ for some constant $C_2>0$. Then
\begin{equation}\label{eq:stima-m-m0}
m(x)\geq \frac{C_1C_2}{\left(1+C_2|x|\right)^{\eta_0}}\geq \frac{C_3}{\left(1+|x|\right)^{\eta_0}}
\end{equation}
for some constant $C_3>0$.

On the other hand, since $\beta \geq \alpha$, we obtain by \eqref{shen-B} that $\tilde V\in B_\infty$. Then, by \cite[Remark 2.9]{shen-1995}, we have
\begin{equation}\label{eq:stima-m-v}
m(x)\geq C_5 \tilde V^{1/2}(x) \ge C_6 |x|^{\frac{\beta}{2}}\left( 1+|x| \right) ^{-\frac{\alpha}{2}}.
\end{equation}
The thesis follows taking into account \eqref{eq:stima-m-m0} and \eqref{eq:stima-m-v}.
\end{proof}

Applying the estimate \eqref{eq:stima-G-shen} and the previous lemma
we obtain the following upper bounds for the Green function $G$.

\begin{lem}
Let $G(x,y)$ denotes the Green function of the Schr\"odinger operator $\Delta -\frac{|x|^\beta}{1+|x|^\alpha}$ and assume that
$\beta>\alpha-2$. Then,
\begin{equation}\label{eq:stima-G}
G(x,y)\leq C_k\, \frac{1}{ 1+|x-y|^k\,\left( 1+|y| \right)^{\frac{\beta-\alpha}{2} k }}\;\frac{1}{|x-y|^{N-2}}, \quad x,y\in\R^N
\end{equation}
for any $k> 0$ and some constant $C_k>0$ depending on $k$.
\end{lem}

Using the above lemma we have the following estimate.
\begin{lem} \label{chius2} Assume that $\alpha > 2$, $N>2$  and $\beta>\alpha-2$. Then
there exists a positive constant $C$ such that for every $0\leq \gamma\leq \beta$ and
$f\in L^p(\R^N)$
\begin{equation}\label{eq:stima-L}
\||x|^\gamma Lf\|_{p}\leq  C\|f\|_{p},
\end{equation}
where $L$ is defined in \eqref{eq:operatore-L}.
\end{lem}
\begin{proof}
Let $\Gamma(x,y)=\frac{G(x,y)}{1+|y|^\alpha}$, $f\in L^p(\R^N)$ and
\[
u(x)=\int_{\R^N}\Gamma (x,y)f(y)dy.
\]
We have to show that
\[
\||x|^\gamma u\|_p\leq C\|f\|_p.
\]

Let us consider the regions $E_1:=\{|x-y|\leq (1+|y|) \}$
and $E_2:=\{|x-y|> (1+|y|) \}$ and write
\[
u(x)=\int_{E_1} \Gamma (x,y)f(y)dy+\int_{E_2} \Gamma (x,y)f(y)dy=:u_1(x)+u_2(x)\;.
\]

In $E_1$ we have
\[
 \frac{1+|x|}{1+|y|}\leq \frac{1+|x-y|+|y|}{1+|y|}\leq 2.
\]
So, by Lemma \ref{pr:stim-m}
\begin{align*}
& \left| |x|^\gamma u_1(x)\right|\leq |x|^\gamma \int_{E_1}\Gamma(x,y)|f(y)|dy \leq
\frac{1+|x|^\beta}{1+|x|^\alpha} \int_{E_1}\frac{1+|x|^\alpha}{1+|y|^\alpha}G(x,y)|f(y)|dy\\
&\quad \leq C (1+|x|)^{\beta-\alpha} \int_{\R^N}G(x,y)|f(y)|dy\leq C m^2(x) \tilde u(x),
\end{align*}
where $\tilde u (x)=\int_{\R^N}G(x,y)|f(y)|dy$.   By \eqref{shen-B} we have $\tilde V\in B_{\frac{N}{2}}$. So applying \cite[Corollary 2.8]{shen-1995}, we obtain $\|m^2\tilde u\|_p\leq C\|f\|_p$ and then
$\| |x|^\gamma u_1\|_p\leq C \|f\|_p$.

In the region $E_2$, we have, by H\"older's inequality,
\begin{align}\label{eq:u-2-holder}
& \left| |x|^\gamma u_2(x) \right|
  \leq |x|^\gamma \int_{E_2}\Gamma(x,y)|f(y)|dy=\int_{E_2}  \left( |x|^\gamma \Gamma(x,y)\right) ^{\frac{1}{p'}}
    \left( |x|^\gamma \Gamma(x,y)\right)^{\frac{1}{p}}|f(y)|dy\nonumber \\
&\quad\leq
  \left(\int_{E_2}|x|^\gamma \Gamma(x,y)dy \right)^{\frac{1}{p'}}
  \left(\int_{E_2}|x|^\gamma \Gamma(x,y)|f(y)|^p dy\right)^{\frac{1}{p}}\;.
\end{align}
We propose to estimate first $\int_{E_2}|x|^\gamma\Gamma(x,y)dy$. In $E_2$ we have $1+|x|\leq  1+|y|+|x-y|\leq 2|x-y|$,
then from \eqref{eq:stima-G} it follows that
\begin{eqnarray*}
|x|^\gamma \Gamma (x,y) &\leq &  |x|^\gamma G(x,y)\\
&\leq & C \frac{1+|x|^\beta}{|x-y|^k\left( 1+|y| \right)^{k\frac{\beta-\alpha}{2}}}\frac{1}{|x-y|^{N-2}}\\
&\leq & C \frac{1}{|x-y|^{k-\beta+N-2}}
    \frac{1}{\left( 1+|y| \right)^{k\frac{\beta-\alpha}{2}}}\;.
\end{eqnarray*}
For every $k>\beta-N+2$, taking into account that $\frac{1}{|x-y|}< \frac{1}{1+|y|}$, we get
\[
|x|^\gamma \Gamma (x,y) \leq
    \frac{1}{\left( 1+|y| \right)^{k\frac{\beta-\alpha+2}{2} +N-2-\beta}}\;.
\]
Since $\beta-\alpha+2>0$ we can choose $k$ such that $\frac{k}{2}(\beta-\alpha +2)+N-2-\beta>N$,
then
\begin{align*}
\int_{E_2}|x|^\gamma \Gamma(x,y)dy\leq
\int_{E_2}|x|^\gamma G(x,y)dy\leq C\int_{\R^N} \frac{1}{(1+|y|)^{\frac{k}{2}(2+\beta-\alpha)+N-2-\beta}}dy<C\,.
\end{align*}
Moreover by the symmetry of $G$  we have
\begin{eqnarray*}
|x|^\gamma \Gamma (x,y) &\leq & |x|^\gamma G(x,y)\\
  &\leq & C\frac{1+|x|^\beta}{|x-y|^k\left( 1+|x| \right)^{k\frac{\beta-\alpha}{2}}}\frac{1}{|x-y|^{N-2}}\\
  &\leq & C \frac{1}{|x-y|^{k-\beta +N-2}}
    \frac{1}{\left( 1+|x| \right)^{k\frac{\beta-\alpha}{2}}}\;.
\end{eqnarray*}
Taking into account that $\frac{1}{|x-y|}\leq 2\frac{1}{1+|x|}$, arguing as above we obtain
\begin{align}\label{eq:vudx}
\int_{E_2}|x|^\gamma \Gamma(x,y)dx
\leq C.
\end{align}
Hence \eqref{eq:u-2-holder} implies
\begin{equation}\label{eq:V-2-u-2}
\left| |x|^\gamma u_2(x) \right|^p \leq C\int_{E_2}|x|^\gamma \Gamma(x,y)|f(y)|^pdy.
\end{equation}
Thus, by \eqref{eq:V-2-u-2} and \eqref{eq:vudx}, we have
\begin{align*}
&\||x|^\gamma u_2\|^p_p\leq C \int _{E_2} \int _{E_2}|x|^\gamma\Gamma(x,y)|f(y)|^p dydx\\
&\quad = C\int _{E_2} |f(y)|^p \left(\int _{E_2}|x|^\gamma \Gamma(x,y)dx\right)dy\leq C\|f\|_p^p\;.
\end{align*}
\end{proof}


We are now ready to show the invertibility of $A_p$ and $D(A_p)\subset D(V)$.

\begin{prop}
 \label{invertibile}
Assume that $N>2,\,\alpha > 2$ and $\beta>\alpha-2$. Then the operator $A_p$ is
closed and invertible.
Moreover there exists $C>0$ such that, for every $0\leq \gamma\leq \beta $, we have
\begin{equation}\label{eq:stima-apriori}
\|\,|\cdot |^{\gamma }u\|_{p}\leq  C\|A_pu\|_{p},\quad \forall u\in D_{p,max}(A)\;.
\end{equation}
\end{prop}
\begin{proof}
Let us first prove the injectivity of $A_p$. Let $u\in D_{p,max}(A)$ such that $A_pu=0$, in particular
$\tilde A_pu=0$. It follows that $u\in D_{p,max}(\tilde A)=D(\Delta)\cap D\left(\frac{|x|^\beta}{1+|x|^\alpha}\right)$, see \cite{okazawa} (see \cite[Theorem 2.5]{lor-rhan}).
Then multiplying $A_pu$ with $u|u|^{p-2}$ and integrating over $\R^N$ we obtain, by \cite{met-spi},
\begin{eqnarray*}
0&=& \int_{\R^N} u|u|^{p-2} \Delta u \,dx-\int_{\R^N}\frac{|x|^\beta}{1+|x|^\alpha}|u|^pdx \\
&=& -(p-1)\int_{\R^N}|u|^{p-2}|\nabla u|^2dx-\int_{\R^N}\frac{|x|^\beta}{1+|x|^\alpha}|u|^pdx,
\end{eqnarray*}
from which we have $u\equiv 0$. So, by \eqref{eq:operatore-L} we obtain the invertibility of $A_p$.\\
By elliptic regularity one deduces that $A_p$ is closed on $D_{p, max}(A)$.
Finally, the estimate \eqref{eq:stima-apriori} follows from \eqref{eq:stima-L}.
\end{proof}

The previous Theorem gives in particular the $A_p$-boundedness of the potential $V$ and the following regularity result.
\begin{cor}\label{w2p}
Assume that $N>2,\,\alpha > 2$ and $\beta>\alpha-2$. Then
\begin{itemize}
\item[(i)] there exists $C>0$ such that for every $u\in D_{p,max}(A)$
\[
\|Vu\|_p\leq C \|A_pu\|_p ;
\]
\item[(ii)] \[
D_{p,max}(A)=\left\{ u\in W^{2,p}(\R^N)\;|\;Au\in L^p(\R^N)\right\}.
\]
\end{itemize}
\end{cor}
%
%
\begin{proof} We have only to prove the inclusion $D_{p,max}(A)\subset \left\{ u\in W^{2,p}(\R^N)\;|\;Au\in L^p(\R^N)\right\}$.
Let $u\in D_{p,max}(A)$. Then, by (i), $Vu\in L^p(\R^N)$ and hence
\[
 \Delta u=\frac{Au+Vu}{1+|x|^\alpha}\in L^p(\R^N).
\]
So, the thesis follows from the Calderon-Zygmund inequality.
\end{proof}

We can now state the main result of this section
\begin{teo} \label{risolvente}
Assume that $N>2$, $\beta>\alpha - 2$ and $\alpha >2$.  Then, $[0,+\infty)\subset \rho(A_p)$ and $(\lambda -A_p)^{-1}$ is a positive operator on $L^p(\R^N)$ for any $\lambda\ge 0$. Moreover, if $f \in L^p(\R^N) \cap
C_0(\R^N)$, then $(\lambda-A_p)^{-1}f=(\lambda-A)^{-1}f$.
\end{teo}
\begin{proof}
Let us first prove that if $0\le \lambda \in \rho(A_p)$, then $(\lambda -A_p)^{-1}$ is a positive operator on $L^p(\R^N)$. To this purpose, take $0\le f\in L^p(\R^N)$ and set $u=(\lambda -A_p)^{-1}f$. Then, by Corollary \ref{w2p}, $u\in D(\tilde A_p)$ and
$$-(\tilde A_p-\lambda q)u=qf=:\tilde f,$$ where $q(x)=\frac{1}{1+|x|^\alpha}$. Since $\tilde A_p$ generates an exponentially stable and positive $C_0$-semigroup $(\tilde T_p(t))_{t\ge 0}$ on $L^p(\R^N)$ (see \cite[Theorem 2.5]{lor-rhan}), it follows that the semigroup $(e^{-t\lambda q}\tilde T_p(t))_{t\ge 0}$ generated by $\tilde A_p-\lambda q$ is positive and exponentially stable. Hence, $$u=(\lambda q-\tilde A_p)^{-1}\tilde f\ge 0.$$
We show that $E=[0,+\infty)\cap \rho(A_p)$ is an non empty open and closed set in $[0,+\infty)$.
\\
By Proposition \ref{invertibile} we have $0 \in \rho(A_p)$ and hence $E\neq \emptyset$. On the other hand, using the above positivity property and the
resolvent equation we have $(\lambda-A_p)^{-1} \le (-A_p)^{-1}=L$ for any $\lambda \in E$ and therefore
\begin{equation} \label{stimapositiva}
\|(\lambda-A_p)^{-1}\| \le \|L\|\;,
\end{equation}
it follows that the operator norm of $(\lambda-A_p)^{-1}$ is bounded in $E$ and consequently $E$ is closed.
Finally, since $\rho(A_p)$ is an open set, it follows that $E$ is open in $[0,+\infty)$. Thus, $E=[0,+\infty)$.

Now in order to show the last statement we may assume $f \in C_c^\infty$, the thesis will follows
by density.
Setting $u:=(\lambda-A_p)^{-1}f$, we obtain, by local elliptic regularity (cf. \cite[Theorem 9.19]{gil-tru}),
that $u \in C^{2+\sigma}_{loc}(\R^N)$ for some
$0<\sigma <1$. On the other hand, $u\in W^{2,p}(\R^N)$, by Corollary \ref{w2p}. If $p\ge \frac{N}{2}$, then by Sobolev's inequality, $u\in L^q(\R^N)$ for all $q\in [p,+\infty)$. In particular, $u\in L^q(\R^N)$ for some $q> \frac{N}{2}$ and hence $Au=-f+\lambda u\in L^q(\R^N)$. Moreover, since $u \in C^{2+\sigma}_{loc}(\R^N)$, it follows that $u\in W^{2,q}_{loc}(\R^N)$. So, $u\in D_{q,max}(A)\subset W^{2,q}(\R^N)\subset C_b(\R^N)$, by Corollary \ref{w2p} and Sobolev's embedding theorem, since $q>\frac{N}{2}$.

Let us now suppose that $p<\frac{N}{2}$. Take the sequence $(r_n)$, defined by $r_n=1/p-2n/N$ for any $n\in\N$, and set $q_n=1/r_n$ for any $n\in\N$. Let $n_0$ be the smallest integer such that $r_{n_0}\le 2/N$ noting that $r_{n_0}>0$.
Then, $u\in D_{p,max}(A)\subset L^{q_1}(\R^N)\cap L^p(\R^N)$,
by the Sobolev embedding theorem. As above we obtain that $u\in D_{q_1,max}(A)\subset L^{q_2}(\R^N)$. Iterating this argument, we deduce that $u\in D_{q_{n_0},max}(A)$. So we can conclude that $u\in C_b(\R^N)$ arguing as in the previous case. Thus,
$Au=-f+\lambda u\in C_b(\R^N)$.
Again, since $u \in C^{2+\sigma}_{loc}(\R^N)$, it follows that $u\in W^{2,q}_{loc}(\R^N)$ for any $q\in (1,+\infty)$. Hence, $u\in D_{max}(A)$. So, by the uniqueness of the solution of the elliptic problem, we have $(\lambda -A_p)^{-1}f=(\lambda -A)^{-1}f$
for any $f\in C_c^\infty(\R^N)$.
%
\end{proof}


\section{Generation of semigroups}

In this section we show that $A_p$ generates an analytic semigroup on $L^p(\R^N),\,1<p<\infty$, provided that $N>2,\,\alpha > 2$ and $\beta>\alpha-2$.

We start by proving a weighted gradient estimate. To this purpose we
need the following covering result from, see \cite[Proposition 6.1]{cup-for}, to prove  a weighted gradient estimate.

\begin{prop} \label{pr:covering} For every $0 \le k <1/2$ there exists a natural number  $\zeta=\zeta (N,k)$ with the following property:
Given $\mathcal{F}= \{B(x,\rho(x))\}_{x\in \R^N}$, where $\rho:\R^N\to \R_+$ is a Lipschitz continuous function with Lipschitz constant $k$.
Then there exists a countable subcovering $\{B(x_n,\rho(x_n))\}_{n\in\N}$ of $\R^N$ such that at most $\zeta$ among the double balls
$\{B(x_n,2\rho(x_n))\}_{n\in \N}$ overlap.
\end{prop}
To prove the main result of this section we need the following weighted gradient estimate.

\begin{lem}\label{lm:stima-gradiente}
Assume that $N>2,\,\alpha >2$ and $\beta >\alpha -2$. Then there exists a constant $C>0$ such that for every $u\in D_{p,max}(A)$ we have
\begin{equation}
\|(1+|x|^{\alpha-1})\nabla u\|_p\leq C (\|A_pu\|_p+\|u\|_p)\;.
\end{equation}
\end{lem}
\begin{proof}
Let $u\in D_{p,max}(A)$.
We fix $x_0\in \R^n$ and choose $\vartheta\in C_c^{\infty }(\R^N)$ such that $0\leq \vartheta \leq 1$, $\vartheta(x)=1$ for $x\in B(1)$ and $\vartheta(x)=0$ for $x\in \R^N \setminus B(2)$. Moreover, we set $\vartheta_\rho(x)=\vartheta \left(\frac{x-x_0}{\rho}\right)$, where $\rho=\frac{1}{4}(1+|x_0|)$.
We apply the well-known inequality
\begin{equation}\label{interpolation}
\|\nabla v\|_{L^p(B(R))}\leq C\|v\|^{1/2}_{L^p(B(R))}
\|\Delta v\|^{1/2}_{L^p(B(R))},\;\;v\in W^{2,p}(B(R))\cap W^{1,p}_0 (B(R)),\;R>0
\end{equation}
to the function $\vartheta_\rho u$ and obtain for every $\varepsilon>0$,
\begin{align*}
&\|(1+|x_0|)^{\alpha -1}\nabla u\|_{L^p(B(x_0,\rho))} \leq \|(1+|x_0|)^{\alpha -1}\nabla (\vartheta _\rho u)\|_{L^p(B(x_0,2\rho))}\\
&\quad    \leq C \|(1+|x_0|)^{\alpha} \Delta (\vartheta _\rho u)\|^{\frac{1}{2}}_{L^p(B(x_0,2\rho))}
      \|(1+|x_0|)^{\alpha -2}\vartheta _\rho u\|^{\frac{1}{2}}_{L^p(B(x_0,2\rho))}\\
& \quad \leq C \left(
    \varepsilon \|(1+|x_0|)^{\alpha} \Delta (\vartheta _\rho u)\|_{L^p(B(x_0,2\rho))}+
  \frac{1}{4\varepsilon}\|(1+|x_0|)^{\alpha -2}\vartheta _\rho u\|_{L^p(B(x_0,2\rho))}
 \right)\\
&\quad \leq C \left(
    \varepsilon \|(1+|x_0|)^{\alpha} \Delta u\|_{L^p(B(x_0,2\rho))}
     +\frac{2M}{\rho}\varepsilon \|(1+|x_0|)^{\alpha}\nabla u\|_{L^p(B(x_0,2\rho))}\right. \\
&\qquad \left.
    +\frac{M}{\rho^2}\| (1+|x_0|)^{\alpha}u\|_{L^p(B(x_0,2\rho))}
	  +\frac{1}{4\varepsilon}\|(1+|x_0|)^{\alpha -2} u\|_{L^p(B(x_0,2\rho))}
 \right)\\
&\quad \leq C \left(
    \varepsilon \|(1+|x_0|)^{\alpha} \Delta u\|_{L^p(B(x_0,2\rho))}
     +8M \varepsilon \|(1+|x_0|)^{\alpha-1}\nabla u\|_{L^p(B(x_0,2\rho))}\right.\\
&\qquad \left.+\left( 16M+\frac{1}{4\varepsilon}\right) \|(1+|x_0|)^{\alpha -2} u\|_{L^p(B(x_0,2\rho))}
      \right)\\
&\quad \leq C(M) \left(
    \varepsilon \|(1+|x_0|)^{\alpha} \Delta u\|_{L^p(B(x_0,2\rho))}
     + \varepsilon \|(1+|x_0|)^{\alpha-1}\nabla u\|_{L^p(B(x_0,2\rho))}\right.\\
&\qquad \left.+ \frac{1}{\varepsilon} \|(1+|x_0|)^{\alpha -2} u\|_{L^p(B(x_0,2\rho))}
      \right),
\end{align*}
where $M=\|\nabla \vartheta \|_{\infty }+\|\Delta \vartheta\|_{\infty }$.
Since  $2\rho=\frac{1}{2}(1+|x_0|)$ we get
\[
\frac{1}{2}(1+|x_0|)\leq 1+|x|\leq \frac{3}{2}(1+|x_0|),\qquad x\in B(x_0,2\rho).
\]
Thus
\begin{align}\label{eq:cover-x0}
&\|(1+|x|)^{\alpha -1}\nabla u\|_{L^p(B(x_0,\rho))}\leq \left(\frac 32 \right)^{\alpha -1}\|(1+|x_0|)^{\alpha -1}\nabla u\|_{L^p(B(x_0,\rho))} \nonumber \\
& \quad \leq C\left(
    \varepsilon \|(1+|x_0|)^{\alpha} \Delta u\|_{L^p(B(x_0,2\rho))}
     + \varepsilon \|(1+|x_0|)^{\alpha-1}\nabla u\|_{L^p(B(x_0,2\rho))}\right.\nonumber \\
&\qquad \left.+\frac{1}{\varepsilon} \|(1+|x_0|)^{\alpha -2}u\|_{L^p(B(x_0,2\rho))}
      \right)\;\nonumber \\
& \quad \leq C \left(
    2^\alpha \varepsilon \|(1+|x|)^{\alpha} \Delta u\|_{L^p(B(x_0,2\rho))}
     +2^{\alpha-1} \varepsilon \|(1+|x|)^{\alpha-1}\nabla u\|_{L^p(B(x_0,2\rho))}\right.\nonumber \\
&\qquad \left.+\frac{2^{\alpha-2}}{\varepsilon} \|(1+|x|)^{\alpha -2}u\|_{L^p(B(x_0,2\rho))}
      \right).
\end{align}
Let $\{B(x_n,\rho(x_n ))\}$ be a countable covering of $\R^N$ as in Proposition \ref{pr:covering}
such that at most $\zeta $ among the double balls $\{B(x_n,2\rho(x_n ))\}$ overlap.

We write \eqref{eq:cover-x0} with $x_0$ replaced by $x_n$ and sum over $n$.
To the limit as n tends to infinity, taking into account the covering result above, we get

we get
\begin{align*}
\|(1+|x|)^{\alpha -1}\nabla u\|_{p} \leq & \ C
    \big( \varepsilon \|(1+|x|)^{\alpha} \Delta u\|_{p}
     +\varepsilon \|(1+|x|)^{\alpha-1}\nabla u\|_{p}\\&+\frac{1}{\varepsilon} \|(1+|x|)^{\alpha -2}u\|_{p}
     \big)\;.\nonumber
\end{align*}
Choosing $\varepsilon$ such that $\varepsilon C<1/2$ we have
\begin{align*}
\frac{1}{2}\|(1+|x|)^{\alpha -1}\nabla u\|_{p} \leq \frac{1}{2} \|(1+|x|)^{\alpha} \Delta u\|_{p}+\frac{C}{\varepsilon} \|(1+|x|)^{\alpha -2}u\|_{p}
     \;.\nonumber
\end{align*}
It follows from Corollary \ref{w2p} that $\||x|^{\alpha-2}u\|_p\leq \|(1+|x|^\beta)u\|_p \leq \|u\|_p+\|Vu\|_p\leq \|u\|_p+C\|A_pu\|_p$ for every $u\in D_{p,max}(A)$ and some $C>0$. Hence,
$$\|(1+|x|)^{\alpha -1}\nabla u\|_{p} \leq  C(\|A_pu\|_p+\|u\|_p)$$ for all $u\in D_{p,max}(A)$. This ends the proof of the lemma.
\end{proof}

The following lemma shows that $C_c^\infty(\R^N)$ is a core for $A_p$.
\begin{lem}
Assume $N>2,\,\alpha >2$ and $\beta>\alpha-2$. The space $C_c^\infty(\R^N)$ is dense in $D_{p,max}(A)$ with respect to the graph norm.
\end{lem}
\begin{proof} Let us first observe that $C_c^\infty(\R^N) $ is dense in
$W_c^{2,p}(\R^N)$ with respect to the operator norm.
Let $u\in W_c^{2,p}(\R^N)$ and consider $u_n=\rho_n\ast u$,
where $\rho_n$ are standard mollifiers.
We have $u_n\in C_c^\infty(\R^N)$,
$u_n\to u$ in $L^p(\R^N)$ and
$D^2u_n\to D^2u$ in $L^p(\R^N)$.
Moreover, $\supp u_n\subset \supp
u+B(1):=K$ for any $n\in \N$. Then
\begin{align*}
&\|A_pu-Au_n\|_p=\|A_pu-Au_n\|_{L^p(K)}\\
&\quad\leq \|(1+|x|^\alpha)\Delta (u-u_n)\|_{L^p(K)}+\||x|^\beta (u-u_n)\|_{L^p(K)}\\
&\quad \leq \|(1+|x|^\alpha)\|_{L^\infty(K)}\|\Delta (u-u_n)\|_{L^p(K)}+\||x|^\beta\|_{L^\infty(K)}\| (u-u_n)\|_{L^p(K)}
\to 0 \text{ as }n\to \infty\;.
\end{align*}

Now, let $u$ in $D_{p,max}(A)$ and let $\eta$ be a
smooth function such that $\eta=1$ in $B(1)$, $\eta=0$ in
$\R^N\setminus B(2)$, $0\leq \eta\leq 1$ and set
$\eta_n(x)=\eta\left(\frac{x}{n}\right)$. Then consider
$u_n=\eta_n u\in W_c^{2,p}(\R^N)$.
First we have
$u_n\to u$ in $L^p(\R^N)$ by dominated convergence.
As regard $A_pu_n$ we have
\begin{align*}
&A_pu_n(x)=(1+|x|^\alpha)\Delta (\eta_n u)(x)-|x|^\beta \eta_n(x) u(x)\\
&\quad =\eta_n(x)A_pu(x)+2(1+|x|^\alpha) \nabla \eta_n(x)\nabla u(x)+(1+|x|^\alpha)\Delta \eta_n (x) u(x)\\
&\quad =\eta_n(x)A_pu(x)+\frac{2}{n}(1+|x|^\alpha) \nabla \eta\left(\frac{x}{n}\right)\nabla u(x)
  +\frac{1}{n^2}(1+|x|^\alpha)\Delta \eta \left(\frac{x}{n}\right) u(x)\;,
\end{align*}
and
$$\eta _nA_pu\to A_pu \qquad \hbox{in}\qquad  L^p(\R^N)$$
by dominated convergence. As regards the lasts terms we
consider that  $\nabla \eta (x/n)$  and $\Delta \eta (x/n)$ can be different from zero only for $n \le
|x| \le 2n$, then we have
$$\frac{1}{n}(1+|x|^{\alpha})\left|\nabla\eta\left(\frac{x}{n}\right)\right| | \nabla u|
\leq C(1+|x|^{\alpha-1})|\nabla u|\chi_{\{n\leq |x|\leq 2n\}},$$
and
$$\frac{1}{n^2}(1+|x|^{\alpha})\left|\Delta \eta\left(\frac{x}{n}\right)\right| |u|
\leq C(1+|x|^{\alpha-2})|u|\chi_{\{n\leq |x|\leq 2n\}}\,.$$
The right hand side tends to $0$ as $n\to\infty$,
since by Proposition \ref{invertibile} and Lemma \ref{lm:stima-gradiente} we have
$\|(1+|x|^{\alpha-2})u\|_p\leq C(\|A_pu\|_p+\|u\|_p)$ and $\|(1+|x|^{\alpha-1})\nabla u\|_p\leq C(\|A_pu\|_p+\|u\|_p)$.
\end{proof}

Let us give now the main result of this section.

\begin{teo}
Assume $N>2,\,\alpha>2$ and $\beta>\alpha-2$. Then the operator $A_p$ with domain $D_{p,max}(A)$ generates an analytic  semigroup in $L^p(\R^N)$.
\end{teo}
\begin{proof}
Let $f\in L^p$, $\rho>0$.
Consider the operator $\widetilde{A_p}:=A_p-\omega$ where
$\omega$ is a constant which will be chosen later. It is known that the elliptic problem
in $L^p(B(\rho))$
\begin{equation}\label{palla}
\left\{
\begin{array}{ll}
\lambda u- \widetilde{A_p}u=f&\text{ in }B(\rho)\\
u=0&\text{ on }\partial B(\rho)\;,
\end{array}
\right.
\end{equation}
admits a unique solution $u_\rho$ in $W^{2,2}(B(\rho))\cap W_0^{1,2}(B(\rho))$ for $\lambda >0$,
(cf. \cite[Theorem 9.15]{gil-tru}).

Let us prove that that $e^{\pm i\theta}\widetilde{A_p}$ is dissipative in $B(\rho)$ for $0\le \theta \le \theta_\alpha$ with suitable $\theta_\alpha \in (0,\frac{\pi}{2}]$. To this purpose
observe that
\[
\widetilde{A_p}u_\rho={\rm div}\left((1+|x|^\alpha)\nabla u_\rho \right) -\alpha |x|^{\alpha-1}\frac{x}{|x|}\nabla u_\rho-|x|^\beta u_\rho-\omega u_\rho\,.
\]
Set $u^\star=\overline u_{\rho}|u_\rho|^{p-2}$ and recall that $a(x)=1+|x|^\alpha$. Multiplying $\widetilde{A_p}u_\rho$ by $u^\star$ and integrating over $B(\rho)$, we obtain
\begin{eqnarray*}
& &\int_{B(\rho)}\widetilde{A_p} u_\rho\, u^\star
dx=-\int_{B(\rho)}a(x)|u_\rho|^{p-4}|Re(\ov{u}_\rho\nabla
u_\rho)|^2dx\\
& & - \int_{B(\rho)}a(x)|u_\rho|^{p-4}|Im(\ov{u}_\rho\nabla
u_\rho)|^2dx- (p-2)\int_{B(\rho)}a(x)|u_\rho|^{p-4}
\ov{u}_\rho\nabla u_\rho Re(\ov{u}_\rho\nabla u_\rho)dx\\
& & -\alpha\int_{B(\rho)}\ov{u}_\rho|u_\rho|^{p-2}|x|^{\alpha-1}\frac{x}{|x|}\nabla u_\rho\,dx
-\int_{B(\rho)} \left(|x|^\beta+\omega\right) |u_\rho|^{p}dx.
\end{eqnarray*}
We note here that the integration by part in the singular case $1<p<2$ is allowed thanks to \cite{met-spi}.
By taking the real and imaginary part of the left and the right hand
side, we have
\begin{eqnarray*}
& & Re\left( \int_{B(\rho)}\widetilde{A_p}u_\rho\, u^\star dx  \right)\\
&=&
  -(p-1)\int_{B(\rho)}a(x)|u_\rho|^{p-4}|Re(\ov{u}_\rho\nabla u_\rho)|^2dx-\int_{B(\rho)}a(x)|u_\rho|^{p-4}|Im(\ov{u}_\rho\nabla u_\rho)|^2dx\\
& &\qquad
    -\alpha\int_{B(\rho)}|u_\rho|^{p-2}|x|^{\alpha-1}\frac{x}{|x|}Re(\ov{u}_\rho\nabla u_\rho)\,dx
    -\int_{B(\rho)} \left(|x|^\beta+\omega\right) |u_\rho|^{p}dx\\
&=& -(p-1)\int_{B(\rho)}a(x)|u_\rho|^{p-4}|Re(\ov{u}_\rho\nabla u_\rho)|^2dx-
    \int_{B(\rho)}a(x)|u_\rho|^{p-4}|Im(\ov{u}_\rho\nabla u_\rho)|^2dx\\
& &\qquad
    -\frac{ \alpha }{p}\int_{B(\rho)}|x|^{\alpha-1}\frac{x}{|x|}\nabla (|u_\rho|^p)\, dx
    -\int_{B(\rho)} \left(|x|^\beta+\omega\right) |u_\rho|^{p}dx\\
&=& -(p-1)\int_{B(\rho)}a(x)|u_\rho|^{p-4}|Re(\ov{u}_\rho\nabla u_\rho)|^2dx
    -\int_{B(\rho)}a(x)|u_\rho|^{p-4}|Im(\ov{u}_\rho\nabla u_\rho)|^2dx\\
& &\qquad + \int_{B(\rho)} \left( \frac{ \alpha(N-2+\alpha) }{p}|x|^{\alpha-2}-|x|^\beta-\omega \right)|u_\rho|^pdx
\end{eqnarray*}
and
\begin{eqnarray*}
Im\bigg(\int_{B(\rho)}\widetilde{A_p}u_\rho\, u^\star
dx\bigg) &=& -(p-2)\int_{B(\rho)}a(x)|u_\rho|^{p-4}Im(\ov{u}_\rho\nabla
u_\rho) Re(\ov{u}_\rho\nabla
u_\rho)\,dx\\
& & \quad -\alpha\int_{B(\rho)}|u_\rho|^{p-2}|x|^{\alpha-1}\frac{x}{|x|}Im(\ov{u}_\rho\nabla u_\rho)\,dx.
\end{eqnarray*}
We can choose $\tilde c>0$ and $\omega>0$ (depending on $\tilde c$) such that $$ \frac{ \alpha(N-2+\alpha) }{p}|x|^{\alpha-2}-|x|^\beta-\omega   \leq -\tilde c |x|^{\alpha-2}.$$
So, we obtain
\begin{eqnarray*}
-Re\left( \int_{B(\rho)}\widetilde{A_p}u_\rho\, u^\star dx  \right)
&\geq &(p-1)\int_{B(\rho)}a(x)|u_\rho|^{p-4}|Re(\ov{u}_\rho\nabla u_\rho)|^2dx
    \\ & &
 \quad +\int_{B(\rho)}a(x)|u_\rho|^{p-4}|Im(\ov{u}_\rho\nabla u_\rho)|^2dx
+\tilde c \int_{B(\rho)}|u_\rho|^{p}|x|^{\alpha-2}dx
\\ & &
=(p-1)B^2+C^2+\tilde c D^2.
\end{eqnarray*}
Moreover,
\begin{eqnarray*}
& & \left|Im \left(\int_{B(\rho)}\widetilde{A_p} u_\rho\, u^\star dx\right)\right|\\
&\leq & |p-2|\left(\int_{B(\rho)}|u_\rho|^{p-4}a(x)|Re(\ov{u}_\rho\nabla u_\rho)|^2 dx\right)^\frac{1}{2}
    \left(\int_{B(\rho)}|u_\rho|^{p-4}a(x)|Im(\ov{u}_\rho\nabla u_\rho)|^2 dx\right)^\frac{1}{2}\\
& &\qquad
  +\alpha\left(\int_{B(\rho)}|u_\rho|^{p-4}|x|^{\alpha}|Im(\ov{u}_\rho\nabla u_\rho)|^2\; dx\right)^\frac{1}{2}\left(\int_{B(\rho)}|u_\rho|^{p}|x|^{\alpha-2}\; dx\right)^\frac{1}{2}\\
&=& |p-2|BC+\alpha CD,
\end{eqnarray*}
where
\begin{align*}
&B^2=\int_{B(\rho) }|u_\rho|^{p-4}a(x)|Re(\ov{u}_\rho\nabla u_\rho)|^2 dx,\\
&C^2=\int_{B(\rho)}|u_\rho|^{p-4}a(x)|Im(\ov{u}_\rho\nabla u_\rho)|^2
dx,\\
&D^2=\int_{B(\rho)}|u_\rho|^{p}|x|^{\alpha-2}\; dx.
\end{align*}

Let us observe that, choosing $\delta^2=\frac{|p-2|^2}{4(p-1)}+\frac{\alpha^2}{4\tilde c}$ (which is independent of $\rho$), we obtain
$$
\left|Im\bigg(\int_{B(\rho)}\widetilde{A_p} u_\rho\, u^\star dx\bigg)\right|
  \leq \delta \left\{-Re\bigg(\int_{B(\rho)}\widetilde{A_p}u_\rho\, u^\star dx\bigg)\right\}.
$$
If $\tan\theta_\alpha=\delta$, then $e^{\pm i\theta}\widetilde{A_p}$ is
dissipative in $B(\rho)$ for $0 \le \theta \le \theta_\alpha$.
From \cite[Theorem I.3.9]{pazy} follows that the problem
(\ref{palla}) has a unique solution $u_\rho$ for every $\lambda \in
\Sigma_\theta ,\,0 \le \theta <\theta_\alpha$ where
$$
\Sigma_\theta=\{\lambda \in \C\setminus \{0\}: |Arg\,  \lambda| <
\pi/2+\theta\}.
$$
Moreover, there exists a constant $C_\theta$ which is independent of $\rho$, such that
\begin{equation} \label{stima}
\|u_\rho\|_{L^p (B(\rho))} \le
\frac{C_\theta}{|\lambda|}\|f\|_{L^p},\quad \lambda \in \Sigma_\theta .
\end{equation}
Let us now fix  $\lambda \in \Sigma_\theta$, with $0<\theta <\theta_\alpha$ and a radius $r>0$.
We apply the interior $L^p$ estimates
(cf. \cite[Theorem 9.11]{gil-tru})  to the functions $u_\rho$ with $\rho >r+1$. So, by
(\ref{stima}) we have
\[
\|u_\rho\|_{W^{2,p}(B(r))}\leq C_1\left( \|\lambda u_\rho-\widetilde{A_p}u_\rho\|_{L^p(B(r+1))}+\|u_\rho\|_{L^p(B(r+1))}\right) \le C_2\|f\|_{L^p}.
\]
Using a weak compactness and a diagonal argument,
we can construct a sequence $(\rho_n) \to \infty$ such that the functions $(u_{\rho_n})$
converge weakly in $W^{2,p}_{loc}$ to a function $u$ which satisfies $\lambda u-\widetilde{A_p}u=f$ and
\begin{equation}
\label{stima-2}
 \|u\|_{p} \le
\frac{C_\theta}{|\lambda|}\|f\|_{p}, \quad \lambda \in \Sigma_\theta .
\end{equation}
Moreover, $u \in D_{p,max}(A_p)$.
We have now only to show
that $\lambda-\widetilde{A_p}$ is invertible on $D_{p,max}(A_p)$ for
$\lambda \in \Sigma_\theta$.
Consider the set
\[
E=\{r>0: \Sigma_\theta \cap C(r) \subset\rho (\widetilde{A_p})\},
\]
where $C(r):=\{\lambda \in \C : |\lambda |<r\}$.
Since, by Theorem \ref{risolvente}, $0$ is in the resolvent set of $\widetilde{A_p}$, then
$R=\sup E>0$.
On the other hand, the norm of
the resolvent is bounded by
$C_\theta/|\lambda|$ in $C(R) \cap \Sigma_\theta$, consequently  it cannot explode on the
boundary of $C(R)$, then $R=\infty$ and this ends the proof of the theorem.
\end{proof}

\begin{os}\label{os-regularity}
Since $A_p$ generates an analytic semigroup $T_p(\cdot)$ on $L^p(\R^N)$ and the semigroups $T_q(\cdot),\,q\in (1,\infty)$ are consistent, see Theorem \ref{risolvente}, one can deduces (as in the proof of \cite[Proposition 2.6]{lor-rhan}) using Corollary \ref{w2p} that $T_p(t)L^p(\R^N)\subset C^{1+\nu}_b(\R^N)$ for any $t>0,\,\nu \in (0,1)$ and any $p\in (1,\infty)$.
\end{os}

We end this section by studying the spectrum of $A_p$. We recall from Proposition \ref{invertibile}
that $$\||x|^\beta u\|_p\le C\|A_pu\|_p,\quad \forall u\in D_{p,max}(A).$$ So, arguing as in \cite{lor-rhan} we obtain the following results
\begin{prop} \label{comp}
Assume $N>2,\,\alpha>2$  and $\beta>\alpha-2$ then
\begin{enumerate}
\item[(i)] the resolvent of $A_p$ is compact in $L^p$;
\item[(ii)] the spectrum of $A_p$ consists of a sequence of negative real eigenvalues which accumulates at $-\infty$. Moreover, $\sigma(A_p)$
is independent of $p$;
\item[(iii)] the semigroup $T_p((\cdot)$ is irreducible, the eigenspace corresponding to the largest eigenvalue $\lambda_0$ of $A$ is one-dimensional and is spanned by strictly positive functions $\psi$,
 which is radial, belongs to $C_b^{1+\nu}(\R^N)\cap C^2(\R^N)$ for any $\nu \in (0,1)$ and tends to $0$ when $|x|\to \infty$.
\end{enumerate}
\end{prop}

\end{document}